\documentclass[12pt]{amsart}
 \usepackage{amscd,amsmath,amsthm,amssymb}
\usepackage{pstcol,pst-plot,pst-3d}%\usepackage[T1]{fontenc}
 \def\NZQ{\Bbb}               % the font for N,Z,Q,R,C

 \def\ZZ{{\NZQ Z}}

  \def\MS{{\mathcal S}}

 \def\ab{{\bold a}}
 
 \def\xb{{\bold x}}

 \def\cb{{\bold c}}

 \def\opn#1#2{\def#1{\operatorname{#2}}} % to make operators
 \opn\chara{char} \opn\length{\ell} \opn\pd{pd} \opn\rk{rk}
 \opn\projdim{proj\,dim} \opn\injdim{inj\,dim} \opn\rank{rank}
 \opn\depth{depth} \opn\grade{grade} \opn\height{height}
 \opn\embdim{emb\,dim} \opn\codim{codim}
 
 \opn\Tr{Tr} \opn\bigrank{big\,rank}
 \opn\superheight{superheight}\opn\lcm{lcm}
 \opn\trdeg{tr\,deg}%\emph{
 \opn\reg{reg} \opn\lreg{lreg} \opn\ini{in} \opn\lpd{lpd}
 \opn\size{size} \opn\sdepth{sdepth}
 \opn\link{link}\opn\fdepth{fdepth}\opn\lex{lex}
 \opn\div{div} \opn\Div{Div} \opn\cl{cl} \opn\Cl{Cl}
 \opn\Spec{Spec} \opn\Supp{Supp} \opn\supp{supp} \opn\Sing{Sing}
 \opn\Ass{Ass} \opn\Min{Min}\opn\Mon{Mon}
 \opn\Ann{Ann} \opn\Rad{Rad} \opn\Soc{Soc}
 \opn\Im{Im} \opn\Ker{Ker} \opn\Coker{Coker} \opn\Am{Am}
 \opn\Hom{Hom} \opn\Tor{Tor} \opn\Ext{Ext} \opn\End{End}
 \opn\Aut{Aut} \opn\id{id}
 
 \opn\nat{nat}
 \opn\pff{pf}%   \pf exists already
 \opn\Pf{Pf} \opn\GL{GL} \opn\SL{SL} \opn\mod{mod} \opn\ord{ord}
 \opn\Gin{Gin} \opn\Hilb{Hilb}\opn\sort{sort}
 \opn\aff{aff} \opn
 \con{conv} \opn\relint{relint} \opn\st{st}
 \opn\lk{lk} \opn\cn{cn} \opn\core{core} \opn\vol{vol}
 \opn\link{link} \opn\star{star}\opn\lex{lex}\opn\set{set}
 \opn\gr{gr}
 
 \def\pot#1#2{#1[\kern-0.28ex[#2]\kern-0.28ex]}

 \opn\dirlim{\underrightarrow{\lim}}
 \opn\inivlim{\underleftarrow{\lim}}
 \let\dirsum=\oplus
 \let\tensor=\otimes
 \let\iso=\cong

 \let\Dirsum=\bigoplus
 
 \let\to=\rightarrow
 
 \def\Implies{\ifmmode\Longrightarrow \else
         \unskip${}\Longrightarrow{}$\ignorespaces\fi}
 \def\implies{\ifmmode\Rightarrow \else
         \unskip${}\Rightarrow{}$\ignorespaces\fi}
 \def\iff{\ifmmode\Longleftrightarrow \else
         \unskip${}\Longleftrightarrow{}$\ignorespaces\fi}

 \let\:=\colon
 \newtheorem{Theorem}{Theorem}[section]
 \newtheorem{Lemma}[Theorem]{Lemma}
 \newtheorem{Corollary}[Theorem]{Corollary}
 \newtheorem{Proposition}[Theorem]{Proposition}
 \newtheorem{Remark}[Theorem]{Remark}

 \let\epsilon\varepsilon
 \let\kappa=\varkappa
 \def\qed{\ifhmode\textqed\fi
       \ifmmode\ifinner\quad\qedsymbol\else\dispqed\fi\fi}
 \def\textqed{\unskip\nobreak\penalty50
        \hskip2em\hbox{}\nobreak\hfil\qedsymbol
        \parfillskip=0pt \finalhyphendemerits=0}
 \def\dispqed{\rlap{\qquad\qedsymbol}}
 
 \opn\dis{dis}
 \def\pnt{{\raise0.5mm\hbox{\large\bf.}}}
 
 \opn\Lex{Lex}

\begin{document}

 \title {Polarization of Koszul cycles with applications to powers of edge ideals of whisker graphs}

 \author {J\"urgen Herzog, Takayuki Hibi and Ayesha Asloob Qureshi}

%\newline

\address{J\"urgen Herzog, Fachbereich Mathematik, Universit\"at Duisburg-Essen, Campus Essen, 45117
Essen, Germany} \email{juergen.herzog@uni-essen.de}

\address{Takayuki Hibi, Department of Pure and Applied Mathematics, Graduate School of Information Science and Technology,
Osaka University, Toyonaka, Osaka 560-0043, Japan}
\email{hibi@math.sci.osaka-u.ac.jp}

\address{Ayesha Asloob Qureshi, The Abdus Salam International Center of Theoretical Physics, Trieste, Italy} \email{ayesqi@gmail.com}

%\thanks{}

% \address{ }\email{}

 \begin{abstract}
In this paper, we introduced the polarization of Koszul cycles and use it to study the depth function of powers of edge ideals of whisker graphs.
 \end{abstract}


 \maketitle

 \section*{Introduction}

 Polarization is a technique to deform an arbitrary monomial ideal $I$ in a polynomial ring $S$ into a squarefree monomial ideal $I^\wp$ in a larger polynomial ring $S^\wp$ such that $S/I$ is a quotient of $S^\wp/ I^\wp$ modulo a regular sequence of linear forms. The polarized ideal $I^\wp$ has a nice property that it has the same graded Betti numbers as $I$. Therefore, many questions regarding monomial ideals can be reduced to the study of squarefree monomial ideals. The fact that $I$ and $I^\wp$ has same graded Betti numbers implies that the corresponding Koszul homology modules of the ideal and its polarization have the same vector-space dimension. Therefore, it is natural to ask whether cycles whose homology classes form a basis of the Koszul homology of $I$ can be naturally lifted to cycles representing a basis for the Koszul homology of $I^\wp$. In Theorem~\ref{main}, it is shown that this is indeed the case.

 In his book \cite[Proposition 6.3.2]{V}, Villarreal uses polarization to give a simple proof of the fact that the edge ideal of a whisker graph is Cohen-Macaulay. Given a finite simple graph $G$ on the vertex set $V(G)= \{x_1, \ldots, x_n\}$ and the edge set $E(G)$. One defines whisker graph $G^*$ of $G$ to be the graph with vertex set $\{x_1, \ldots, x_n, y_1, \ldots, y_n\}$ and edge set $E(G) \cup \{ \{x_i, y_i\} : \;  i=1, \ldots, n\}$. By using the results of Section~\ref{polarization}, one easily sees that the homology classes of the cycles
 \begin{eqnarray}\label{basis1}
 x_{i_1} \ldots x_{i_k}  e_{j_1} \wedge e_{j_{n-k}}\wedge f_{i_1} \wedge \ldots \wedge f_{i_k}
 \end{eqnarray}
with $\MS = \{i_1, \ldots, i_k\}$ a maximal independent set of $G$ and $\{j_1, \ldots, j_{n-k}\}=V(G) \setminus \MS$, form a basis of the Koszul homology $H_n(x_1, \ldots, x_n, y_1, \ldots, y_n ; S^*/I(G^*))$. Here $e_1, \ldots, e_n, f_1, \ldots, f_n$ is a $S^*$-basis of free module $K_1(x_1, \ldots, x_n, y_1, \ldots, y_n; S^*/I(G^*))$ with $\partial (e_i) = x_i$ and $\partial (f_j)= y_j$. A basis cycle as described in (\ref{basis1}) is used in Section~\ref{whisker graphs} in the study of the powers of edge ideals of whisker graphs.

The homological and algebraic behavior of powers of an ideal has been subject of many research papers in recent years. In particular, the nature of the depth function $f(k)= \depth(S/I^k)$ of a graded ideal $I$ in a polynomial ring $S$ is still quite mysterious. While it is known by a classical result of Brodmann \cite{Br1} that $f(k)$ for $k \gg 0$ is constant, the behavior of $f(k)$ is not so well understood for initial values of $k$. In \cite{HH1}, it is shown that any non-decreasing bounded integer function $f(k)$ is the depth function of a suitable monomial ideal and it is conjectured that $f(k)$ can be any convergent nonnegative integer valued function. In support of this conjecture, it was shown in \cite{BHH} that $f(k)$ may have arbitrarily many local maxima. On the other hand, it seems that the depth function for the edge ideals behave more tamely. In particular, it is expected that the depth function of an edge ideal is a non-increasing function. This is indicated by the fact that edge ideals satisfy the persistence property for the associated prime ideals of their powers, as shown in \cite{CMS}. Interesting lower bounds for the depth function of an edge ideal have been achieved by Morey \cite{M}. On the other hand, even for simple graphs like a line graph or a cycle, the precise depth function is unknown.

In this paper we give an upper bound for the depth function for any connected whisker graph. In fact we show in Theorem~\ref{whisker} that for any connected graph $G$ on the vertex set $[n]$, we have $\depth (S^*/I(G^*)^k) \leq n-k+1$ for $k= 1, \ldots , n$. It can be shown by examples that this upper bound is no longer valid if we drop the assumption that $G$ is connected. For connected graph this upper bound is obtained by constructing suitable non-vanishing homology classes for the Koszul homology of the powers of $I(G^*)$. The cycles representing these non-vanishing homology classes are obtained as products of certain 1-cycles and an $(n-1)$-cycle which is defined via an independent set of $G$. For showing that the homology of this product of cycles in the corresponding homology group is non-vanishing, we use a combinational fact proved in Proposition~\ref{gamma} which says that any connected graph admits a friendly independent set in the sense as described in this proposition. By using results from \cite{CMS} and \cite {EH2}, we show in Corollary~\ref{limit} that  $\depth(S^*/I(G^*)^k) =1$ for $k\geq n$ if $G$ is bipartite and  $\depth(S^*/I(G^*)^k) =0$ for $k\geq n$ if $G$ is non-bipartite.

The upper bound for the depth of the powers of a whisker graph given by our Theorem~\ref{whisker} is not always attained. The simplest examples for such case are the whisker graphs of a $3$-cycle or $4$-cycle. On the other hand, Villarreal \cite[Proposition 6.3.7]{V}, has shown that $\depth(S^*/I(G^*)^2) \geq n-1$ if $G$ is tree (or even a forest)  on the vertex set $[n]$. In Theorem~\ref{tree}, we extend the result of Villarreal and show that for any forest $G$ one has $\depth(S^*/I(G^*)^k) \geq n-k+1$ for $k=1, \ldots , n$. Together with Theorem~\ref{whisker} we conclude that for any tree $G$ we have $\depth(S^*/I(G^*)^k) = n-k+1$ for $k=1, \ldots, n$.

%Together we our theorem, it follows that  $\depth(S'/I(G')^2) = n-1$ for any tree $G$. There is computational evidence that $\depth(S'/I(G')^k)=n-k+1 $ for any tree $G$.

 \section{Polarization of Koszul cycles} \label{polarization}

 Let $K$ be a field and $I\subset S=K[x_1,\ldots,_n]$ a monomial ideal in the polynomial ring $S$. We denote as usual by $G(I)$ the unique minimal set of monomial generators  of $I$. If $u=x_1^{a_1}\cdots x_n ^{a_n}$ is a monomial, we call $\ab=(a_1,\ldots,a_n)$ the multi-degree of $u$ and set $\deg_{x_i}u=a_i$ for all $i$.

 Let $c_i=\max\{\deg_{x_i} u\:\; u\in G(I)\}$ for $i=1,\ldots,n$, and let $S^{\wp}$ be the polynomial ring over $K$ in the variables $x_{ij}$, $i=1,\ldots,n$,  $j=1,\ldots.c_i$. The {\em polarization}  of $I$ is the squarefree monomial $I^\wp\subset S^\wp$ generated by the monomials $u^\wp$ with $u\in G(I)$ where for $u=x_1^{a_1}\cdots x_n^{a_n}$ one sets
 \[
 u^\wp=\prod_{i=1,\ldots,n}\prod_{j=1,\ldots,a_i} x_{ij}.
\]

We extend this polarization operation to elements in the Koszul complex. Let $K(\xb;I)$ be the Koszul complex of the sequence $\xb=x_1,\ldots,x_n$ with values in $I$.  Recall that $K_i(\xb)=\bigwedge^iF$ where $F=\Dirsum_j^nSe_j$ and where $\partial e_j=x_j$ for $j=1,\ldots,n$, and that $K(\xb;I)=K(\xb)\tensor I$. Thus an element of $K_i(\xb;I)$ is of the form
\[
\sum_{J}f_Je_J,
\]
where the sum is taken over all ordered sets $J=\{j_1<j_2<\cdots <j_i\}$ of cardinality $i$, where $f_J\in I$ and where $e_J=e_{j_1}\wedge e_{j_2}\wedge \cdots \wedge e_{j_i}$.

Next we consider the Koszul complex $K(\xb^\wp; I^\wp)$. Here $\xb^\wp$ is the sequence
\[
x_{11},x_{12},\ldots,x_{1c_1},x_{21},\ldots, x_{2c_2},\ldots, x_{n1},\ldots, x_{nc_n},
\]
and $K_i(\xb^\wp)=\bigwedge^i G$ where $G=\Dirsum_{i=1,\ldots,n}\Dirsum_{j=1,\ldots,c_i}S^{\wp}e_{ij}$.

We call an element $u_Je_J$ a {\em monomial} of $K(\xb;I)$ if $u_J$ is a monomial. We set
\[
\deg_{x_i}(u_J e_J)= \deg_{x_i}u_J + \delta_j ,
\]
where
\[
\delta_j= \left\{ \begin{array}{ll}
       1, & \;\textnormal{if $j\in J$}, \\ 0, & \;\text{otherwise.}
        \end{array} \right.
\]
and call
\[
\deg(u_J e_J) = (\deg_{x_1}(u_J e_J), \ldots, \deg_{x_n} (u_J e_J) )
\]
the multi-degree of $u_Je_J$.

For any monomial $u_J e_J$ of multi-degree $\leq \cb$ (componentwise) where $\cb= (c_1, \ldots, c_n)$, we define the  polarization of $u_Je_J$ to be the monomial
\[
(u_Je_J)^\wp=u_J^\wp e_{j_1a_{j_1}+1}\wedge e_{j_2a_{j_2}+1}\wedge \cdots \wedge e_{j_ia_{j_i}+1},
\]
in $K(\xb^{\wp};I^\wp)$, where  $J=\{j_1<j_2<\cdots <j_i\}$, and $a_i = \deg_{x_i} u_J$.

We  extend this polarization operator to an arbitrary multi-homogeneous element $f=\sum_J\lambda_Ju_Je_J$, $\lambda_J\in K$,  of multi-degree $\leq \cb$ , by setting
\[
f^\wp= \sum_J\lambda_J(u_Je_J)^\wp.
\]
If follows from \cite[Theorem 3.1]{BHbook} that any non-vanishing homology class of $H_i(\xb;I)$ can be represented by a multi-homogeneous cycle $z=\sum_J\lambda_Ju_Je_J$ in $K_i(\xb;I)$ with $\deg z \leq \cb$. Thus the polarization of such cycles is defined.

\medskip

The following example demonstrate the polarization of cycles: let $I=(x_1^2 x_2, x_1 x_2^2)$. Then $z= x_1x_2^2 e_1 - x_1^2x_2 e_2$ is a cycle in $K_1(x_1, x_2;I)$, and $z^\wp = x_{11}x_{21} x_{22} e_{12} - x_{11}x_{12}x_{21}e_{22}$.
\medskip

With the notation introduced, we have

\begin{Theorem} \label{main}
Let $I \subset S=K[x_1, \ldots, x_n]$ be a monomial ideal and let $\cb=(c_1, \ldots, c_n)$  be the integer vector with $c_i=\max\{\deg_{x_i} u\:\; u\in G(I)\}$ for $i=1, \ldots, n$. Let $z_1, \ldots, z_r$ be multi-homogeneous cycles with multi-degree $\leq \cb$, whose  homology classes form a $K$-basis of $H_i(\xb;I)$. Then  the homology classes of the cycles $z_1^\wp, \ldots, z_r^\wp$ form  a $K$-basis of $H_i(\xb^\wp; I^\wp)$.
\end{Theorem}

The theorem will be a consequence of the following

\begin{Lemma}
\label{comparison}
Let $M$ be a finitely graded $S$-module, and assume that $x_1$ is a non zero-divisor modulo $M$. Then there is a natural isomorphism
\[
\varphi: H_{i}(x_1,\ldots,x_n;M)\to H_i(x_2,\ldots,x_n;\bar{M}),
\]
where $\bar{M}$ is the $\bar{S} = S/x_1S$-module $M/x_1M$. This isomorphism is given as follows: let $z \in Z_i(x_1,\ldots,x_n;M)$ and write $z=e_1\wedge z_0+z_1$ with $z_1\in K_i(x_2,\ldots,x_n;M)$. Then the homology class $[z]\in H_{i}(x_1,\ldots,x_n;M)$ is mapped to $[\bar{z}_1]\in H_{i}(x_2,\ldots,x_n;\bar{M})$, where $\bar{z}_1$ is obtained from $z_1$ by taking the residue classes of the coefficients of $z_1$ modulo $x_1$.
\end{Lemma}

\begin{proof}
Observe that $\bar{z_1}$ is indeed a cycle in $K(x_2, \ldots x_n; \bar{M})$, because $0= x_1 z_0 - e_1 \wedge \partial z_0 + \partial z_1$. From this equation it follows that $e_1 \wedge \partial z_0 =0$ and hence $\partial \bar{z_1} =0$. Next we show that $\varphi$ is well defined. Let $z$ be as in the statement and let $w=z+\partial b$ where $b \in K_{i+1}(x_1, \ldots, x_n;M)$. Let $b= e_1 \wedge b_0 + b_1$ with $b_1 \in K_{i+1}(x_2, \ldots, x_n; M)$. Then $w= e_1\wedge w_0 + w_1 $ where $w_1 = z_1 + x_1 b_0 + \partial b_1$. We have to show that $[\bar{w}_1] = [\bar{z}_1]$. But this is obvious, because $\bar{w}_1 = \bar{z}_1 + \partial \bar{b}_1$, so that $\bar{w}_1$ and $\bar{z}_1$ differ only by a boundary in $K_{i}(x_2, \ldots, x_n;\bar{M})$.

Since $H_i (x_1, \ldots, x_n ; M) \iso \Tor_i^S (K;M)$ and $H_i (x_2, \dots, x_n; \bar{M}) \iso \Tor_i^{\bar{S}} (K, \bar{M})$, we conclude that $\dim_K H_i (x_1, \ldots, x_n ; M) = \dim_K H_i (x_2, \dots, x_n; \bar{M})$. Indeed, since $x_1$ is a non-zero on $M$, the graded minimal free resolution of $\bar{M}$ is obtained from the graded minimal free resolution of $M$ be reduction modulo $x_1$. This implies that $\dim_K \Tor_i^S (K;M) = \dim_K  \Tor_i^{\bar{S}} (K, \bar{M})$. Hence, in order to prove that $\varphi$ is an isomorphism, it suffices to show that $\varphi$ is surjective.

Let $[v] \in H_i (x_2, \ldots, x_n ; \bar{M})$. There exists $z_1 \in K_i (x_2, \ldots, x_n; M)$ with $\bar{z_1} = v$. It follows that $\partial z_1 = -x_1 z_0$ for some $z_0 \in K_{i-1} (x_2, \ldots, x_n;M)$. Since $0= \partial^2 z_1 = -x_1 \partial z_0$, we see that $\partial z_0 =0$. Now we set $z= e_1 \wedge z_0 + z_1$. Then $z$ is a cycle and $\varphi [z] = [v] $.
\end{proof}

\begin{proof}[Proof of Theorem~\ref{main}]
Fix an integer $1 \leq i \leq n$. For each $u \in G(I)$ we define

\[
u'= \left\{ \begin{array}{ll}
       (u/x_i)y , & \;\textnormal{if $x_i^2 | u$}, \\ u, & \;\text{otherwise.}
        \end{array} \right.
\]
The element $u'$ is called the {\em $1$-step polarization} of $u$ with respect to the variable $x_i$, and the ideal
$I' = (\{u'| u \in G(I)\})$ is called a 1-step polarization of $I$. Obviously, the (complete) polarization of $I$ can be obtained by a sequence of 1-step polarization.

Let $I'$ be the 1-step polarization of $I$ with respect to $x_i$. Without loss of generality, we may assume that $i=1$.  We consider the Koszul complex $K(y,x_1, \ldots, x_n; I')= (\bigwedge H) \tensor I'$ where $H$ is the free $S[y]$-module with basis $f, e_1 \ldots, e_n$ and where $\partial f = y$ and $\partial e_j = x_j$ for $j=1, \ldots, n$. Let $z= \sum_{J} \lambda_J u_J e_J $ be a multi-homogenous cycle of $K_i(x_1, \ldots, x_n;I)$ with $\deg z \leq \cb$ whose homology class is non-zero.

We set $z' = \sum_{J} \lambda_J (u_J e_J)'$, where
\[
(u_J e_J)'= \left\{ \begin{array}{ll}
       u_J e_J , & \;\textnormal{if $x_1 \nmid u_J$}, \\ u_{J}'e_J, & \;\text{if $x_1|u_J$ and $1 \notin J$},
       \\ u_{J}e_{J}', & \;\text{if $x_1|u_J$ and $1 \in J$}.
        \end{array} \right.
\]

Here $e_{J}'$ is obtained from $e_J$ by replacing the factor $e_1$ in $e_J$  by $f$.
\medskip

As an example we consider again the cycle $z= x_1x_2^2 e_1 - x_1^2x_2 e_2$ in $K_1(x_1,x_2;I)$ where $I=(x_1^2 x_2, x_1x_2^2)$. Then $z' = x_1 x_2^2 f - x_1 y x_2 e_2$.

\medskip

We claim that $z'$ is a cycle in $K_i(y, x_1, \ldots, x_n; I')$, and that the map
\[
 H_i (y, x_1, \ldots, x_n;I') \rightarrow H_i( x_1, \ldots, x_n; I), \quad [z'] \mapsto [z]
\]
is an isomorphism. From this claim the theorem follows by induction on the number of 1-step polarization which are required to obtain the polarization $I^{\wp}$ of $I$.

\medskip
Proof of the claim: we first show that $z'$ is a cycle. We first discuss the case when $\deg_{x_1} z \leq 1$.

By the definition of $(u_J e_J)'$, we have $(u_J e_J)' = u_J e_J$, for all $J$. It shows that $z=z'$ and hence $z'$ is a cycle.

Now we discuss the case when $\deg_{x_1}z > 1$. Let $z=e_1\wedge z_0+z_1$  and $z'= f \wedge z'_0 +z'_1$ with $z_1\in K_i(x_2,\ldots,x_n;I)$ and $z'_1\in K_i(x_1,\ldots,x_n;I')$. Moreover,  $z_0 = \sum_{1 \in J} \lambda_J u_J e_{J \setminus \{1\}}$ and $z_1 = \sum_{1 \notin J} \lambda_J u_J e_J$. From the definition of $z'$ we see that $z'_0 = z_0$ and $z'_1  = \sum_{1\notin J} \lambda _J u'_J e_J$ where $u'_J = y u_J / x_1 $. It implies that $z'_1 = (y/x_1 ) z_1$. By applying $\partial$ on $z'$, we get $\partial (z') = y z_0 + \partial (z'_1) = y z_0 + (y/x_1) \partial (z_1)$. It  shows that $x_1\partial (z') = y \partial (z)  = 0$. Hence $\partial (z') = 0$.

We first observe that $y-x_1$ is a non-zero divisor on $S[y]/I'$ and that $I' / (y-x_1)I' = I$. Therefore, by Lemma~\ref{comparison}, there exists the isomorphism $\varphi : H_i(y, x_1, \ldots, x_n) = H_i(y-x_1, x_1, \ldots, x_n; I') \rightarrow H_i (x_1, \ldots, x_n; I)$. Thus it remains to be shown that  $\varphi([z']) = [z]$. Applying the Lemma~\ref{comparison}, we write $z'= (f-e_1) \wedge w_0 +w_1$.

By definition,
\begin{eqnarray*}
z'&=& \sum_{x_1 \nmid u_J} u_Je_J + \sum_{x_1 | u_J, 1 \notin J} u'_Je_J + \sum_{x_1|u_J, 1 \in J} u_J f \wedge e_{J \setminus \{1\}} \\
&=& \sum_{x_1 \nmid u_J} u_Je_J + \sum_{x_1 | u_J, 1 \notin J} u'_Je_J + \sum_{x_1|u_J, 1 \notin J} u_J (f-e_1) \wedge e_{J \setminus \{1\}} +  \sum_{x_1|u_J, 1 \in J} u_J  e_{J}.
\end{eqnarray*}

Therefore,
\begin{eqnarray*}
w_1&=& \sum_{x_1 \nmid u_J} u_Je_J + \sum_{x_1 | u_J, 1 \notin J} u'_Je_J +  \sum_{x_1|u_J, 1 \in J} u_J  e_{J}.
\end{eqnarray*}

From this it follows that $\bar{w}_1 = z$,  which by the definition of $\varphi$ implies that $\varphi ([z']) = [z]$, as desired.
\end{proof}

\begin{Corollary}\label{polarize}
Let $I\subset S$ be a monomial ideal as in Theorem~\ref{main}. Let $z_1, \ldots, z_r$ be multi-homogeneous cycles with multi-degree $\leq \cb$, whose  homology classes form a $K$-basis of $H_i(\xb;S/I)$ for $i \geq 1$. Then  the homology classes of the cycles $z_1^\wp, \ldots, z_r^\wp$ form  a $K$-basis of $H_i(\xb^\wp; S^\wp/I^\wp)$.
\end{Corollary}

\begin{proof}
We notice that for $i \geq 1$ there is an isomorphism $\varphi: H_i(\xb^\wp;S^\wp/I^\wp) \rightarrow H_{i+1}(\xb^\wp;I^\wp)$ with $\varphi ([z]) = [\partial(w)]$ and $w \in K(\xb^\wp; S^\wp)$ such that $z=w+I^\wp K(\xb^\wp;S^\wp)$. Since $\partial(f^\wp) = (\partial(f))^\wp$ for any multi-homogenous element $f \in K(\xb; S^\wp)$ with $\deg f \leq \cb$, the desired conclusion follows.
\end{proof}

As an example for the polarization of Koszul cycles, we consider whisker graphs. Let $G$ be a finite simple graph on the vertex set $[n]=\{1, \ldots, n\}$. The {\em whisker graph} $G^*$ of $G$ is the graph with the vertex set $V(G^*)=\{1, \ldots, n\} \cup \{1', \ldots, n'\}$ and the edge set $E(G^*)=E(G) \cup \{ \{1, 1'\}, \{2, 2'\}, \ldots, \{n, n'\} \}$.

\medskip
Figure~\ref{one} displays the whisker graph of the graph $G$ with edges $\{1,2\}, \{2,3\}, \{3,4\}$ and $\{4,2\}$.

\begin{figure}[hbt] \begin{center} \label{one}
\psset{unit=0.6cm}\begin{pspicture}(0.5,1)(4.5,5) \pspolygon(2,2)(3,3.71)(4,2)\psline(3,3.71)(3,5.2)\psline(0.6,1.1)(2,2)\psline(4,2)(4,3.49) \psline(2,2)(2,3.49) \psline(0.6,1.1)(0.6,2.59) \rput(0.6,2.59){$\bullet$} \rput(2,3.49){$\bullet$}\rput(4,3.49){$\bullet$}\rput(2,2){$\bullet$}\rput(3,3.71){$\bullet$}\rput(4,2){$\bullet$}\rput(3,5.2){$\bullet$}\rput(0.6,1.1){$\bullet$}\rput(0.6,3.1){$1'$}\rput(2,4){$2'$}\rput(4.25,4){$3'$}\rput(2,1.5){$2$}\rput(3.25,4){$4$}
\rput(4,1.5){$3$}\rput(3.1,5.7){$4'$}\rput(0.6,0.6){$1$}
\end{pspicture}
\end{center}
\caption{}
\label{example}\end{figure}

Let $K$ be a field. The edge ideal $I(G)$ of $G$ is the monomial ideal in $S=K[x_1, \ldots, x_n]$ generated by the monomials $x_ix_j$ with $\{i,j\} \in E(G)$. We consider the edge ideal $I(G^*)$ of the whisker graph $G^*$ of $G$ as the monomial ideal in $S^*=K[x_1, \ldots, x_n, y_1, \ldots, y_n]$ with $I(G^*) = I_G + (\{ x_k y_k |  k \in [n]  \})$.

Next, we let $J(G) = ( I(G) , x_1^2 , \ldots, x_n^2 )$. Then, obviously, $I(G^*) = J(G)^{\wp}$, where for simplicity  we set $x_i = x_{i1} , y_i = x_{i2}$, for $i= 1, \ldots, n$. For the polarized Koszul complex of $K(x_1, \ldots, x_n; I(G))$ we use the notation $e_i = e_{i1}$ and $f_i = e_{i2}$. Given a cycle $\sum_J \lambda_J u_J e_J \in K(x_1, \ldots, x_n; J(G))$ representing a non-zero homology class, the polarized cycle in $K(x_1, \ldots, x_n, y_1, \ldots, y_n ; I(G^*))$ is given as $\sum_J \lambda_j u_j e_{J'}$ where $e_{J'} $ is obtained from $e_J$ by replacing $e_{j}$ for $j \in J$ by $f_j$ if $x_j |u$.

Note that $H_n(x_1, \ldots, x_n; J(G))$ is minimally generated by the homology classes $[u e_1\wedge \ldots \wedge e_n]$ with $u = x_{i_1}\ldots x_{i_k}$ where $\{i_1, \ldots, i_k\}$ is a maximal independent set of $G$. Recall that a subset $\MS \subset V(G)$ is an {\em independent set} of $G$ if $\{i,j\} \notin E(G)$ for all $i,j \in \MS$. The set $\MS$ is called a maximal independent set if $\MS \cup \{k\}$ is not independent for all $k \notin V(G) \setminus \MS$.

 It follows from Corollary \ref{polarize}, that the elements
 \begin{eqnarray}\label{basis}
 x_{i_1} \ldots x_{i_k}  e_{j_1} \wedge e_{j_{n-k}}\wedge f_{i_1} \wedge \ldots \wedge f_{i_k}
 \end{eqnarray}
form a basis of $H_n(x_1, \ldots, x_n, y_i, \ldots, y_n ; S^*/I(G^*))$ where $\MS = \{i_1, \ldots, i_k\}$ is a maximal independent set of $G$ and $\{j_1, \ldots, j_{n-k}\}=V(G) \setminus \MS$.

\section{Powers of whisker graphs}\label{whisker graphs}

In this section, we want to study the powers of whisker graphs. For the formulation of the main result we introduce the following concept. Let $G$ be a finite simple graph on $[n]$, and let $\MS$ be a maximal independent subset of $V(G)$. We define the graph $\Gamma_{\MS}(G)$ with vertex $V(\Gamma_{\MS}(G)) =  \MS$ and $\{i,j\} \in E(\Gamma_{\MS}(G))$ if and only if there exists $k \in V(G) \setminus \MS$ such that $\{i,k\}, \{j,k\}
 \in E(G)$.

\begin{Proposition}\label{gamma}
Let $G$ be a finite simple connected graph. Then there exists an independent set $\MS$ such that $\Gamma_{\MS}(G)$ is connected.
\end{Proposition}

\begin{proof}
Let $\Delta(G)$ be the clique complex of $G$ with cliques $F_1, \ldots, F_r$. We are going to construct the independent set $\MS$ of $G$ as follows. Let $v_1 \in V(F_1)$. We may assume that $v_1 \in V(F_i)$ for $i = 1, \ldots, t$ and $v_1 \notin V(F_i)$ for $i>t$. If $t=r$, then we are done. Assume that $t <r$. Since $G$ is connected, there exists  $F_i$ with $i>t$, say $F_{t+1}$, such that $V(F_{t+1}) \cap V(F_j) \neq \emptyset $ for some $j \leq t$. Since $F_{t+1}$ is a maximal clique,  $V(F_{t+1}) \not\subset \bigcup_{i=1}^t V(F_i) $ because otherwise $v_1 \in V(F_{t+1})$, a contradiction. Hence, we may choose $v_2 \in V(F_{t+1} )\setminus \bigcup_{i=1}^t V(F_i) $. We may assume that $v_2 \in V(F_i)$ for $i=t+1, \ldots, s$ and does not belong to any other clique. If $s=r$, then $\Gamma_{\MS}(G) $ is a line graph with vertex set $ \{v_1, v_2\}$. Indeed, $\{v_1, v_2\} \notin E(G)$  because the set of neighbors of $v_1$ is equal to $\bigcup_{i=1}^t F_i$ and $v_2 \notin \bigcup_{i=1}^t F_i $. On the other hand, if $k \in V(F_{t+1}) \cap V(F_j)$. then $\{v_1,k\}, \{v_2,k\} \in E(G)$. Therefore, $\{v_1, v_2\} \in E(\Gamma_{\MS}(G))$.

Consider all $F_j$ for $j >s$ such that $V(F_j) \subset \bigcup_{i=1}^t V(F_i)$. We may assume that it is the case for $F_{s+1}, \ldots, F_k$. If $k=r$, then $\{v_1, v_2\}$ is an independent set for $G$, and we are done. If $k<r$, then since $G$ is connected there exists a clique $F_i$, say $F_{k+1}$, such that $V(F_{k+1}) \cap V(F_j) \neq \emptyset $ for some $j<k$ and $V(F_{k+1}) \not\subset \bigcup_{i=1}^s V(F_i) (=  \bigcup_{i=1}^k V(F_i) ) $. We choose $v_3 \in V(F_{k+1}) \setminus \bigcup_{i=1}^s V(F_i)$. If $j<t$ then $\{v_1, v_3\} $ will be an edge of $\Gamma_{\MS}(G)$, and if $t+1\leq j\leq s$, then $\{v_2,v_3\}$ will be an edge of $\Gamma_{\MS}(G)$. Proceeding this way, we obtain the desired independent set $\MS$ of $G$ such that $\Gamma_{\MS}(G)$ is connected.
\end{proof}

We call an independent set $\MS$ of $G$ {\em friendly} if it satisfies the condition that $\Gamma_{\MS}(G)$ is connected. For example, if we consider the line graph $L$ on vertex set $[4]$ with edges $\{\{1,2\}, \{2,3\}, \{3,4\}\}$. Then $\MS= \{1,3\}$ is a friendly independent set  of $L$ while $\{1,4\}$ is not a friendly independent set of $L$.

\begin{Theorem}\label{whisker}
Let $G$ be a finite simple connected graph on vertex set $[n]$, and $G^*$ be the whisker graph of $G$. Furthermore, let $I(G^*)\subset S^*=K[x_1, \ldots, x_n, y_1, \ldots, y_n]$ be the edge ideal of $G^*$. Then
\[
\depth (S^*/I(G^*)^k) \leq n-k+1 , \text{\; for \;} k= 1, \ldots, n.
\]
\end{Theorem}

\begin{proof}
Let $M$ be an $S^*$-module and consider the Koszul complex
\[
K(M) = K(x_1, \ldots, x_n, y_1, \ldots, y_n;M)
\]
with $K_1(M) = \Dirsum_{i=1}^n M e_i \dirsum \Dirsum_{j=1}^n M f_j$ and $\partial e_i = x_i$ and $\partial f_j = y_j$, for all $i,j$.

We first show that
\[
H_{2n-2} (I(G^*)^n) \neq 0.
\]

This will imply that $\depth (S^*/ I(G^*)^n) \leq 1$. To see that the above Koszul homology does not vanish, we proceed as follows.

By Proposition~\ref{gamma} we may choose a friendly independent set $\MS$ of $G$ with $|\MS| =s$. Since $\Gamma_{\MS}(G)$ is connected, there exists a spanning tree $T$ of $\Gamma_{\MS}(G)$ with $s-1$ edges $\alpha_1, \ldots, \alpha_{s-1}$. We may assume that $\alpha_1, \ldots, \alpha_{s-1}$ is a leaf order for $T$. In other words, the following conditions are satisfied: (i) $\alpha_1$ has a free vertex in $T$, (ii) for each $j>1$, $ \alpha_j \cap \alpha_i \neq \emptyset$ for some $i<j$ and $\alpha_j$ has a free vertex in the tree $T_j= \alpha_1, \ldots, \alpha_{j}$. Now, we label the vertices of $T$ inductively as follows: $1$ is the free vertex of $\alpha_1$ in $T_1$ and the other vertex in $T_1$ is given the label $2$. Suppose, the labeling of $T_{j-1}$ is defined. Then we give the new vertex of $T_j$, the label $j+1$. Then $\alpha_1=\{1,2\}$ and for each $j>1$, $\alpha_j=\{i_j,j+1\}$, where $\{i_j\}= \alpha_j \cap \alpha_i$.

The following Figure~\ref{two}, gives an example of such a labeling.

\begin{figure}[hbt]\begin{center}\label{two}
\psset{unit=0.6cm}\begin{pspicture}(0.5,1.5)(4.5,3)\psline(2,2)(4,2)\psline(0,1)(2,2) \psline(0,3)(2,2)\psline(4,2)(6,1)\psline(4,2)(6,3)\rput(2,2){$\bullet$}\rput(4,2){$\bullet$}\rput(0,1){$\bullet$}
\rput(0,3){$\bullet$}\rput(6,1){$\bullet$}\rput(6,3){$\bullet$}\rput(2,1.5){$2$}\rput(4,1.5){$4$}\rput(-0.4,0.9){$1$}
\rput(-0.4,3.1){$3$}\rput(6.4,0.9){$5$}\rput(6.4,3.1){$6$}\rput(3,1.6){$\alpha_3$}\rput(1.3,1.1){$\alpha_1$}
\rput(1.3,2.9){$\alpha_2$}\rput(4.9,2.9){$\alpha_4$}\rput(5,1.1){$\alpha_5$}
\end{pspicture}
\end{center}
\caption{}
\label{example}\end{figure}

According to our labeling of $T$, we have $\MS= \{1, \ldots, s\}$. By definition of $\Gamma_{\MS}(G)$, there exists for each edge $\alpha_j=\{i_j,j+1\} \in E(T)$, a vertex $v_j \in \{s+1, \ldots, n\}$ such that $\{i_j,v_{j}\}, \{v_j,j+1\} \in E(G)$. Then $z_j= x_{i_j}x_{v_j} e_{j+1} - x_{j+1} x_{v_j} e_{i_j}$ is a cycle belonging to $Z_1(I(G^*))$.

Furthermore, for each $k \in \{s+1, \ldots, n\}$, we choose $j_k \in \MS$ such that $\{k, j_k\} \in E(G)$. Then $z_k = x_k x_{j_k} f_k - x_k y_k e_{j_k} $ is a cycle belonging to $Z_1(I(G^*))$. This gives $n-s$ such cycles.

Let
\[
c= \prod _{i =1}^s x_i  e_{s+1} \wedge \ldots \wedge e_n \wedge f_1 \wedge \ldots \wedge f_s.
\]
Note that by (\ref{basis}), $c$ is a cycle in $Z_n(S^*/I(G^*))$ whose homology class $[c]$ in $H_n(S^*/I(G^*))$ is non-zero. In particular, $[\partial (c)]$ is non-zero homology class in $H_{n-1}(I(G^*))$.

Let
\[
a=c \wedge  z_1\wedge \ldots \wedge z_{s-1} \wedge z_{s+1} \wedge \ldots \wedge z_n.
\]
Observe that $a \in K_{2n-1}(I(G^*)^{n-1})$. We set $z= \partial (a)$. Then $z \in Z_{2n-2} (I(G^*)^n)$. Indeed, $z = \partial (c) \wedge z_1\wedge \ldots \wedge z_{s-1} \wedge z_{s+1} \wedge \ldots \wedge z_n$, and it has coefficients in $I(G^*)^n$ because $\partial (c)$ and each $z_i$ has coefficients in $I(G^*)$.

 We claim that $[z]$ is a non-zero homology class in $H_{2n-2} (I(G^*)^n)$. To prove the claim, we show that $z$ is not a boundary, that is, there does not exist any $b \in K_{2n-1}(I(G^*)^n)$ such that $z= \partial b$. On contrary, assume that such $b$ exists. Then $\partial (b) = \partial (a) =z$ implies $\partial ( a-b) = 0$ which gives $a-b \in Z_{2n-1}(I(G^*)^{n-1} )$. Then, there exists $b'  \in K_{2n}(S^*)$ such that $\partial (b')= a-b$ where $b' = v e_1 \wedge \ldots \wedge e_n \wedge f_1 \wedge \ldots \wedge f_n$ and $v$ is a monomial in $S^*$ because all cycles under consideration are $\ZZ^{2n}$-graded.

Note that
\[
 w= ( \prod _{i=1}^s x_i  ) (\prod_{k=s+1}^n x_k x_{j_k} )  (\prod_{j=1}^{s-1}  x_{i_j} x_{v_j}) e_2 \wedge \cdots \wedge e_{n} \wedge f_1 \wedge \cdots \wedge f_n
 \]
with $i_j, j_k \in \MS$, $k, v_j \in \{s+1, \ldots, n\}$  is a non-zero term of $a$ and it is not cancelled by any other term of $a$ because the product $ e_2 \wedge \cdots \wedge e_n \wedge f_1 \wedge \cdots \wedge f_n$ appears only once in the expansion of $a$. To see this, consider
\begin{eqnarray}\label{c}
c \wedge z_{s+1} \wedge \cdots \wedge z_n =  ( \prod _{i=1}^s x_i  ) (\prod_{k=s+1}^n x_k x_{j_k} )  e_{s+1} \wedge \cdots \wedge e_{n}  \wedge f_1 \wedge \cdots \wedge f_n.
\end{eqnarray}

Therefore, it follows that the term $w$ appears only once in the expansion of $a$ if $e_2 \wedge \cdots \wedge e_{s-1}$ appears only once in the expansion of $z_1 \wedge \cdots \wedge z_{s-1}$. Now to see this, we write $z_j=g_j - h_j$, where $g_j=x_{i_j}x_{v_j} e_{j+1} $ and $h_j=x_{j+1} x_{v_j} e_{i_j}$ for $j= 1, \ldots, s-1$. Note that for $1 \leq t \leq s-1$ the wedge product $z_1\wedge \cdots \wedge z_{t}$ is a linear combination of $g_{i_1}\wedge \cdots \wedge g_{i_k} \wedge h_{j_1}\wedge \cdots \wedge h_{j_t}$ with $\{i_1, \ldots, i_k\}\cup\{ j_1, \ldots, j_t\}= \{1, \ldots, t\}$. We prove by induction on $t$ that among these terms $g_1 \wedge \cdots \wedge g_t$ is the only term that does not contain $e_1$. For $t=1$, the assertion is trivial. Now let $t>1$ and assume that th sonly term that does not contain $e_1$ is $g_1 \wedge \cdots \wedge g_{t-1}$. Then, the only terms of $z_1 \wedge \cdots \wedge z_t$ which do not contain $e_1$ are either
$g_1 \wedge \cdots \wedge g_{t}$ or $g_1 \wedge \cdots \wedge g_{t-1} \wedge h_t$. However, by the definition of the cycles $z_j$ given in terms of the tree $T$ it follows that $i_t = \{2, \ldots, t-1\}$. Therefore, $g_1 \wedge \cdots \wedge g_{t-1} \wedge h_t=0$.

\medskip
Next, we show that $a  \in K_{2n-1}(I(G^*)^{n-1}) \setminus K_{2n-1} (I(G^*)^n)$. For this it suffice to show that $w   \in K_{2n-1}(I(G^*)^{n-1}) \setminus K_{2n-1} (I(G^*)^n)$, because $w$ is a non-zero term of $a$ which does not cancel against any other term in $a$, as we have just seen. In fact, $ ( \prod _{i=1}^s x_i  ) (\prod_{k=s+1}^n x_k x_{j_k} )  (\prod_{j=1}^{s-1}  x_{i_j} x_{v_j}) $ which is coefficient of $w$, there are $n-1$ terms with indices in $\{s+1, \ldots, n\}$ and $n+s-1$ terms with indices in $\MS=\{1, \ldots, s\}$. Since $\MS$ is a maximal independent set, this implies that $w$ contains a product of exactly $n-1$ generator of $I(G^*)$.

Since all coefficients of $b=a-\partial(b')$ are in $I(G^*)^n$ and the term $w$ which appears in the expansion of $a$ does not have coefficient in $I(G^*)^{n}$, $w$ must be cancelled by some term of $\partial (b')$. This gives
\[
v x_1 e_2 \wedge \cdots  \wedge e_{n} f_1 \wedge \ldots \wedge f_n = ( \prod _{i=1}^s x_i  ) (\prod_{k=s+1}^n x_k x_{j_k} )  (\prod_{j=1}^{s-1}  x_{i_j} x_{v_j}) e_2 \wedge \cdots \wedge e_{n} \wedge f_1 \wedge \cdots \wedge f_n,
\]
which implies
\[
v= ( \prod _{i=2}^s x_i  ) (\prod_{k=s+1}^n x_k x_{j_k} )  (\prod_{j=1}^{s-1}  x_{i_j} x_{v_j})  \in I(G^*)^{n-1}.
\]

The coefficient of the term $v y_n e_1 \wedge \ldots \wedge e_n f_1 \wedge \cdots \wedge f_{n-1} $ which appears in the expansion of $\partial(b')$ does not belong to $I(G^*)^{n}$ because $x_n$ is the only neighbor of $y_n$.  Also the term $v y_n e_1 \wedge \ldots \wedge e_n f_1 \wedge \cdots \wedge f_{n-1} $ is not cancelled by any of the terms of $a$ because from (\ref{c}) we can see that all terms of $a$ contain the wedge product $f_1\wedge \cdots \wedge f_n$ as a factor. Hence, our assumption that $z$ is a boundary leads us to contradiction.

For simplicity of notation, we set $z'_i=z_i$ for $i = 1, \ldots, s-1$ and $z'_i=z_{i+1}$ for $i=s+1, \ldots, n-1$. Note that $\partial(c)\wedge z'_1 \wedge \cdots \wedge z'_{k-1} \in Z_{n+k-2}(I(G^*)^k)$. We claim that this cycle is not a boundary in $K(I(G^*)^k)$. This then implies that $\depth (S^*/I(G^*)^k) \leq n+k-1$ since $H_{n+k-1}(S^*/I(G^*)^k) \iso H_{n+k-2} (I(G^*)^k) \neq 0$.

In order to prove the claim, assume that there exists $b \in K_{n+k-1}(I(G^*)^k)$ such that $\partial (b) = \partial (c) \wedge z'_1 \wedge \cdots \wedge z'_{k-1}$. Let $b'=b \wedge z'_k \wedge \cdots \wedge z'_{n-1} $. Then $b' \in K_{2n-1}(I(G^*)^n)$ and $\partial(b') =  \partial (b) \wedge z'_k \wedge \cdots \wedge z'_{n-1} = z$, a contradiction.
\end{proof}

Our hypothesis of Theorem~\ref{whisker} which requires that $G$ is connected is needed. For example, if we take the disconnected graph $G$ on vertex $[4]$ with edge $\{\{1,2\}, \{3,4\}\}$, then $\depth(S^*/I(G^*)^4) = 2$.

\begin{Remark}{\em
Let $I$ be an arbitrary monomial ideal in $K[x_1, \ldots, x_n, y_1, \ldots, y_n]$. In \cite[Theorem 3.3]{HQ}, it is shown that $\depth (S/I^k) \leq 2n-k+1$ for $k=1, \ldots, r$, where $r< 2n$ is a number depending on $I$. Comparing this result with our Theorem~\ref{whisker}, where $I$ is the edge ideal of a whisker graph, our bound is about half of the bound which is valid for general monomial ideals.}
\end{Remark}

\begin{Corollary}\label{limit}
Let $G$ be a finite simple connected graph on vertex set $[n]$, $G^*$ be the whisker graph of $G$, and $I(G^*)\subset S^*$ be the edge ideal of $G^*$. If $G$ is bipartite, then $\depth (S^*/I(G^*)^k) =1$ for all $k \geq n$, and if $G$ is non-bipartite, then $\depth (S^*/I(G^*)^k) =0$ for all $k\geq n$.
\end{Corollary}

\begin{proof}
Suppose first that $G$ is bipartite. Then $G^*$ is bipartite as well. It follows from \cite[Theorem 5.9]{SVV} that the $\depth (S^*/I(G^*)^k )\geq 1$ for all $k$. Thus our Theorem~\ref{whisker} implies that $\depth(S^*/I(G^*)^n)=1$. On the other hand, since the Rees ring $R(I(G^*))$ of $I(G^*)$ is Cohen-Macaulay  (see for example \cite[Corollary 5.20]{EH}), the result of Eisenbud and Huneke \cite[Proposition 3.3]{EH2}  yields the desired conclusion.

Now let $G$ be a non-bipartite graph. It follows from \cite[Corollary 4.3]{CMS}, applied to our case, that $\Ass (S^*/I(G^*)^k) = \Ass(S^*/I(G^*)^{n})$ for all $k \geq n$. On the other hand, since $G$ is non-bipartite, we know by \cite[Corollary 3.4]{CMS} that  $\depth (S^*/I(G^*)^k) =0$ for $k \gg 0$. This implies that $\depth (S^*/I(G^*)^k) =0$ for all $k\geq n$.
\end{proof}

In general the upper bounds  for the depth of the powers of the edge ideal of a whisker graph given in Theorem~\ref{whisker} are not attained. For example, if $G$ is a 3-cycle then $\depth(S^*/I(G^*))=3 $ and $\depth(S^*/I(G^*)^k)=0$ for all $k \geq 2$. Even if $G$ is bipartite, this bound is not attained. For example, if $G$ is a 4-cycle, then $\depth(S^*/I(G^*))=4$, $\depth(S^*/I(G^*)^2)=3$ and $\depth(S^*/I(G^*)^k) =1$ for $k \geq 3$.

On the other hand, Villarreal showed \cite[Proposition 6.3.7]{V}, that if $G$ is a forest then $\depth(S^*/I(G^*)^2) \geq n-1$. Together with our Theorem~\ref{whisker} it follows that $\depth(S^*/I(G^*)^2) = n-1$.  By using the arguments applied in Villarreal's proof, we now show more generally

\begin{Theorem} \label{tree}
\label{whiskertree}
Let $G$ be a forest on $n$ vertices and let
\[
I=I(G^*)\subset S^*=K[x_1,\dots,x_n,y_1,\ldots,y_n]
\]
be the edge ideal of $G^*$. Then
\[
\depth(S^*/I^k) \geq n-k+1 \quad  \text{for $k=1,\ldots,n$}.
\]
\end{Theorem}

\begin{proof}
We show this by induction on $k + n$. If $k+n=1$, then either $k=1$ or $n=1$. If $k=1$, then $\depth(S^*/I^k)=n$ since $I$ is a Cohen-Macaulay of height $n$ and for $n=1$ the assertion is trivial.

Let $x_n$ be a free vertex of the forest $G$ with $\{x_{n-1}, x_n\} \in E(G)$. Following the notation used in the proof  \cite[Proposition 6.3.7]{V}, we denote by $J$ the ideal which is obtained by $I$ by substituting $x_n=0$ and by $L$ the ideal which obtained from $J$ by substituting $x_{n-1} = 0$.  Furthermore, we set $K=(J^k, x_{n-1}x_n, x_ny_n)$. Since $(J^k, x_{n-1}x_n):x_n = (L^k, x_{n-1})$, we obtain the exact sequence
\begin{eqnarray*}\label{depth1}
0 \rightarrow S^*/ (L^k, x_{n-1}) \rightarrow S^* / ( {J^k, x_{n-1}x_n} )\rightarrow S^* / (J^k, x_n) \rightarrow 0 .
\end{eqnarray*}

Since $J$ is edge of a whisker forest on $n-1$ vertices and $L$ is the edge ideal of the whisker forest on $n-2$ vertices, our induction hypothesis implies that
\[
\depth (S^*/ (L^k, x_{n-1})) \geq n-k+2  \text{ and } \depth (S^* / (J^k, x_n)) \geq n-k+1.
\]
This implies that

\begin{eqnarray}\label{depth2}
\depth (S^* / (J^k, x_{n-1}x_n )) \geq n-k+1.
\end{eqnarray}

Since $(J^k, x_{n-1}x_n): x_ny_n = (L^k, x_{n-1})$, we obtain the exact sequence

\begin{eqnarray}\label{depth3}
0 \rightarrow S^*/ (L^k, x_{n-1}) \rightarrow S^* / ( {J^k, x_{n-1}x_n} )\rightarrow S^* / K \rightarrow 0 .
\end{eqnarray}

From (\ref{depth2}) and (\ref{depth3}), we obtain

\begin{eqnarray}\label{depth4}
\depth (S^* / K) \geq n-k+1.
\end{eqnarray}

Note that $(I^k, x_n y_n) = (J, x_{n-1}x_n)^k + (x_n y_n)$. Therefore, $(I^k , x_n y_n) : x_{n-1} x_n = (J, x_{n-1} x_n)^k : x_{n-1}x_n + (y_n)$. Since $(J, x_{n-1} x_n)$ is the edge ideal of the graph for which $\{n-1, n\}$ is an edge with free vertex $n$, it follows by result of Morey \cite[Lemma 2.10]{M} that  $(J, x_{n-1} x_n)^k : x_{n-1}x_n = (J, x_{n-1} x_n)^{k-1}$. Therefore, altogether we have that $(I^k, x_ny_n): x_{n-1}x_n = (J, x_{n-1}x_n)^{k-1} + (y_n)$. Thus, we obtain the exact sequence

\begin{eqnarray}\label{depth5}
0 \rightarrow S^*/ ((J, x_{n-1x_n})^{k-1} + (y_n) ) \rightarrow S^* / ( {I^k, x_n y_n} )\rightarrow S^* / K \rightarrow 0 .
\end{eqnarray}

We claim that for $k=2, \ldots, n$ we have $\depth (S^*/ (J, x_{n-1}x_n)^{k-1} ) \geq n-k+2$. Therefore, (\ref{depth4}) and (\ref{depth5}) implies

\begin{eqnarray}\label{depth6}
\depth (S^* / ({I^k, x_n y_n}) ) \geq n-k+1.
\end{eqnarray}

By using (\ref{depth6}) and our induction hypothesis, the exact sequence
\begin{eqnarray*}
0 \rightarrow S^*/ I^{k-1} \rightarrow S^* / I^k \rightarrow S^* / (I^k, x_n y_n) \rightarrow 0 .
\end{eqnarray*}
yields that $\depth (S^*/ I^k) \geq n-k+1$ and proves our theorem.

It remains to prove the claim. For that we use induction on $k$. For $k=2$, this inequality is shown in the proof of \cite[Proposition 6.3.7]{V}. Suppose that $k>2$. Since $(J, x_{n-1}x_n)$ is the edge ideal of a tree with free vertex $n$, we may apply \cite[Lemma 2.10]{M}, and obtain the exact sequence
\begin{eqnarray*}\label{depth7}
0 \rightarrow S^*/ (J, x_{n-1x_n})^{k-2}  \rightarrow S^* / ( J,  x_{n-1} x_n )^{k-1} \rightarrow S^* / (J^{k-1}, x_{n-1}x_n) \rightarrow 0 .
\end{eqnarray*}

 By our induction hypothesis, $\depth ( S^*/ (J, x_{n-1x_n})^{k-2} ) \geq n-k+3$ and (\ref{depth2}) applied for $k-1$ yields $\depth ( S^*/ (J, x_{n-1x_n})^{k-2}  ) \geq n-k+2$. Therefore, it follows that $\depth (S^* / ( J,  x_{n-1} x_n )^{k-1} ) \geq n-k+2$, as desired.
\end{proof}

By combining Theorem~\ref{whisker} and Theorem~\ref{tree}, we obtain
\begin{Corollary}
Let $G$ be a tree. Then
\[
\depth(S^*/I(G^*)^k) = n-k+1 \quad  \text{for $k=1,\ldots,n$}.
\]
\end{Corollary}

More generally, we expect that if $G$ is a forest with $m$ connected components, then
\[
\depth(S^*/I(G^*)^k) = \left\{ \begin{array}{ll}
          n-k+1, &  \text{if $k\leq n-m+1$}, \\
         m, &\text{if  $k \geq n-m+1$}.        \end{array} \right.
\]

{}
 \end{document}